\documentclass[11pt,a4paper]{scrartcl}
\usepackage[numbers,square]{natbib}
\usepackage{microtype}
\usepackage[font=small,labelfont=bf]{caption}
\usepackage[
  a4paper,
  left=3cm,
  right=3cm,
  top=3cm,
  bottom=3cm
]{geometry}
\usepackage{graphicx}
\graphicspath{{./}}
\usepackage[utf8]{inputenc}
\usepackage[T1]{fontenc}
\usepackage[english]{babel}
\usepackage{amsmath,amssymb,amsthm}
\usepackage{mathrsfs}
\usepackage{mathtools}
\usepackage{hyperref}

\hypersetup{
  colorlinks,
  linkcolor={blue},
  citecolor={blue},
  urlcolor={blue}
}

\usepackage{macros}

\numberwithin{equation}{section}
\allowdisplaybreaks

\theoremstyle{plain}
\newtheorem{theorem}{Theorem}[section]
\newtheorem{lemma}{Lemma}[section]
\newtheorem{proposition}{Proposition}[section]

\theoremstyle{definition}
\newtheorem{definition}{Definition}[section]

\theoremstyle{remark}
\newtheorem{remark}{Remark}[section]

\newenvironment{proofof}[1]{\begin{proof}[Proof of #1]}{\end{proof}}
\newcommand{\acknowledgements}{\paragraph{Acknowledgments.}}

\begin{document}

\title{Mod-$\phi$ convergence for random variables with cyclotomic generating functions}
\author{
  Christoph Th\"ale\thanks{Faculty of Mathematics, Ruhr University Bochum, Universit\"atsstra\ss e 150, 44801 Bochum, Germany. E-mail: \href{mailto:christoph.thaele@rub.de}{christoph.thaele@rub.de}.}
  \qquad
  Philipp Tuchel\thanks{Faculty of Mathematics, Ruhr University Bochum, Universit\"atsstra\ss e 150, 44801 Bochum, Germany. E-mail: \href{mailto:philipp.tuchel@rub.de}{philipp.tuchel@rub.de}.}
}
\date{}
\maketitle

\begin{abstract}
  \small We consider a sequence of discrete random variables whose generating functions are polynomials only having roots on the unit circle in the complex plane. For such sequences we prove, under mild assumptions, mod-$\phi$ convergence with explicit rate and limiting function. This allows us to recover, in a unified way, classical central limit theorems as well as finer asymptotic expansions. The general theory is applied to random permutations and random partitions.

  \medspace
\end{abstract}

\section{Introduction}

For each \(N \in \mathbb{N}\), let \(X_{N}\) be a random variable taking values in the finite set \(\{0,1,\dots,N\}\).  Its probability generating function is
$$
P_{N}(z) := \mathbb{E}[\,z^{X_{N}}\,] = \sum_{k=0}^{N} \mathbb{P}(X_{N}=k)\,z^{k}, 
\qquad z \in \mathbb{C}.
$$  
Assuming, without loss of generality, that \(\mathbb{P}(X_{N}=N)>0\), the generating function \(P_{N}\) is a polynomial of degree \(N\), whose complex roots are denoted by 
$$
\varrho_{1}^{(N)},\,\varrho_{2}^{(N)},\,\dots,\,\varrho_{N}^{(N)}.
$$
The study of such sequences of random variables via the analytic behavior of their generating polynomials and the geometric location of their zeros goes back at least to the work of Harper \cite{Harper} on the asymptotic normality of the Stirling numbers, and has since become a vibrant area of research.  The underlying philosophy is that the location of the zeros \(\varrho_{j}^{(N)}\) encodes information about the distributional behavior of \(X_{N}\) as \(N \to \infty\).  

Two particular structural scenarios for the zero-sets have proved especially fruitful:
\begin{itemize}
	\item[(i)] The case in which all the roots \(\varrho_{j}^{(N)}\) are real and nonpositive.  In this regime one may exploit the theory of real-rootedness and related factorization results to derive central limit theorems for \(X_{N}\).  
	\item[(ii)] The case in which all the roots lie on the unit-circle in the complex plane, i.e.\  
	\[
	\varrho_{j}^{(N)} \in \{z\in\mathbb{C}\;:\;|z|=1\}\text{ for }j=1,\dots,N.
	\]
	Here the so-called root-unitary behavior permits a different line of attack. Namely, one may exploit factorizations into cyclotomic polynomials, symmetries of the coefficients, or Fourier-analytic methods to capture fluctuations of \(X_{N}\).  
\end{itemize}
It is in fact rather remarkable that a large number of sequences of discrete random variables arising in combinatorial enumeration or statistical mechanics fall into one of these two broadly defined classes.  In the first case (nonpositive real roots) one typically sees classical limit behaviors such as Gaussian fluctuations, which are controlled by the variance growth only.  In the second case (unit-circle roots) one  sees more subtle phenomena, such as the necessity of controlling higher moments in order to obtain normal (or non-normal) limits, see \cite{heerten2024probabilistic,hwang2015limit,RednossThaele}. We refer to the surveys \cite{BilleySwanson,BrandenSurvey} for further details and background material on these two classes of random variables and generating polynomials.

In the present work we focus on the second, cyclotomic or root-unitary, case, where all zeros of the probability generating polynomial lie on the unit circle. This situation is particularly rich, since many classical generating functions from combinatorics and algebra factor into products of cyclotomic terms. Typical examples include inversion statistics of random permutations, length distributions in reflection groups, and generating functions of plane partitions, see \cite{BilleySwanson,heerten2024probabilistic,hwang2015limit,RednossThaele} for additional references. To analyse such models, we employ the framework of mod-$\phi$ convergence, which has not been used in this area before. It provides a powerful analytic method to describe refined limit theorems beyond the classical central limit regime in a unified way. In contrast to ordinary convergence in distribution, mod-$\phi$ convergence captures precise exponential and oscillatory corrections to the Gaussian approximation, and thereby allows to identify higher-order fluctuations, local limits, and moderate deviation estimates. We refer to \cite{feray2018mod} for a detailed overview of the limit theorems that follow once mod-$\phi$ convergence is established and a general introduction to the topic. Formally, a sequence of random variables $(X_N)_{N\in\N}$ is said to converge mod-$\phi$ with parameter sequence $(w_N)_{N\in\N}$, normalization function $\phi$ and limiting function $\psi$ if
$$
\lim_{N\to\infty}{\E[e^{z X_N}] \over e^{w_N \phi(z)}}  = \psi(z),
$$
where $\phi$ and $\psi$ are analytic functions and the convergence is supposed to take place locally uniformly on a suitable complex domain. Roughly speaking, the term $e^{w_N \phi(z)}$ encodes the leading exponential growth, while the analytic function $\psi(z)$ captures the limiting residual fluctuations. It is instructive to think of $X_N$ as $Y_N+Z+O(1/N)$, where $Y_N$ is a 'random variable' with moment generating function $e^{w_N\phi}$ and $Z$ is an independent 'random variable' with moment generating function $\psi$ -- even though there might be no true random variables having $e^{w_N\phi}$ or $\psi$ as their moment generating functions, see the discussion in \cite{kabluchko2024mod}.

A large class of cyclotomic generating functions can be written in the explicit product form
$$
P_N(z)= \Big(\prod_{j=1}^N\frac{a_j}{b_j}\Big)\Big(\prod_{j=1}^N\frac{1-z^{b_j}}{1-z^{a_j}}\Big),
$$
for sequences $(a_j)_{j=1,\ldots,N}$ and $(b_j)_{j=1,\ldots,N}$ of natural numbers with $b_j\ge a_j$, see \cite{hwang2015limit}. In this setting, we study the rescaled moment generating functions
$$
\E\big[e^{z X_N/s_N}\big]
= \Big(\prod_{j=1}^{N}\frac{a_j}{b_j}\Big)
\Big(\prod_{j=1}^{N}\frac{1 - e^{z b_j/s_N}}{1 - e^{z a_j/s_N}}\Big),
$$
and establish a general mod-$\phi$ convergence result under mild regularity assumptions.
In particular, our main theorem, Theorem~\ref{thm:modphi-mixed} below, provides an explicit description of the analytic function $\psi$ and the normalization function $\phi$. It unifies several classical models within a single framework. For instance, in the classical case of the number of inversions in a uniform random permutation, we obtain mod-$\phi$ convergence with parameters
$$
w_N=N,\qquad \psi(z) = \sqrt{\frac{e^z - 1}{z}}\, e^{-z/2},
\qquad
\phi(z) = \int_0^1 \log\Big(\frac{e^{tz}-1}{tz}\Big) dt.
$$
This framework implies, for example, the asymptotic equivalence
$$
\P(X_N = \E[X_N] + k) \sim {\psi(z^*(x))\over\sqrt{2\pi N\phi''(z^*(x))}}\,e^{-NI(x)},\qquad x={k\over N},\qquad N\to\infty,
$$
where for $x>0$, $z^*(x)$ is the unique solution of $\phi'(z^*(x))=x$ and $I(x)=xz^*(x)-\phi(z^*(x))$, see \cite[Theorem 3.4]{feray2018mod}.

The remaining parts of this paper are structured as follows. In Section \ref{sec:MainResult} we present the main result of this paper, Theorem \ref{thm:modphi-mixed}. We also rephrase there a formal definition of the concept of mod-$\phi$ convergence we work with. Applications of Theorem \ref{thm:modphi-mixed} to random permutations and random partitions are discussed in Section \ref{sec: applications}. The proof of Theorem \ref{thm:modphi-mixed} is the content of Section \ref{sec: proofs}. The final Section \ref{sec:proofs_app} contains the proofs of the applications.

\section{Main result}\label{sec:MainResult}

Let $(X_N)_{N\in\N}$ be a sequence of nonnegative integer-valued random variables with probability generating functions $P_N(z)$ of the form
$$
P_N(z) = \Big(\prod_{j=1}^N\frac{a_{N,j}}{b_{N,j}}\Big) \Big(\prod_{j=1}^N \frac{1-z^{b_{N,j}}}{1-z^{a_{N,j}}}\Big),
$$
where $a_{N,j}$ and $b_{N,j}$ are natural numbers with $b_{N,j} \ge a_{N,j}$ for all $j=1, \dots, N$ and  $N\in\N$. It is not automatic that a function of this form is a probability generating function. Indeed, although $P_N(1)=1$, its power-series coefficients need not all be nonnegative. A sufficient condition is that $a_{N,j}|b_{N,j}$ for every $j=1,\ldots,N$, in which case each factor is itsself a probability generating function. Probability generating functions of the above form have all roots located on the unit circle of the complex plane, see \cite{BilleySwanson,heerten2024probabilistic,hwang2015limit}. We are interested in the moment generating function $\mathbb{E}[e^{z X_N}]$ of $X_N$. Looking at the rescaled random variables ${X_N}/{s_N}$, for some scaling sequence $s_N\to\infty$, we get the object of interest in this work, namely
\begin{align}\label{eq:DefKN}
\E\big[e^{z X_N/s_N} \big] = \Big(\prod_{j=1}^N \frac{a_{N,j}}{b_{N,j}} \Big)\Big(\prod_{j=1}^N \frac{1-e^{\frac{z}{s_N} b_{N,j}}}{1-e^{\frac{z}{s_N} a_{N,j}}}\Big),
\end{align}
assuming that $z \frac{a_{N,j}}{s_N} \notin 2\pi i\,\mathbb{Z}$ for $j=1,\ldots,N$. The product representation in \eqref{eq:DefKN} is understood via analytic continuation at its (removable) singularities $z\,a_{N,j}/s_N\in 2\pi i\,\Z$.

Next, we recall the concept of mod-$\phi$ convergence for sequences of random variables. In this paper, we follow the approach of  \cite{kabluchko2024mod}, which is close but not equivalent to the one in \cite{feray2018mod}. Still, the central theorems of \cite{feray2018mod} continue to hold in this set-up. 

\begin{definition}\label{def: mod_phi_convergence}
	Let $(X_N)_{N\in\mathbb{N}}$ be a sequence of real-valued random variables with moment generating functions $\mathbb{E}[e^{zX_N}]$ existing in some strip
	$\mathscr{H} = \{z \in \mathbb{C} : \Re z \in (\beta_-, \beta_+)\}$ with $-\infty \leq \beta_- < \beta_+ \leq +\infty$. Suppose we are given
	\begin{itemize}
	\item[(i)] a sequence $(w_N)_{N\in\mathbb{N}}$ of positive numbers with $\lim\limits_{N\to\infty} w_N = +\infty$;
	\item[(ii)] an open, connected set $\mathscr{D} \subset \mathscr{H}$ containing the interval $(\beta_-, \beta_+)$;
	\item[(iii)] an analytic function $\phi : \mathscr{D} \to \mathbb{C}$ whose restriction to the interval $(\beta_-, \beta_+)$ is real-valued and strictly convex;
	\item[(iv)] an analytic function $\psi(z) : \mathscr{D} \to \mathbb{C}$ which does not vanish on $\mathscr{D} \cap \mathbb{R}$.
	\end{itemize}
	If
	\begin{align*}
		\lim_{N\to\infty} \frac{\mathbb{E}e^{zX_N}}{e^{w_N \phi(z)}} = \psi(z),
	\end{align*}
	locally uniform on $\mathscr{D}$, then we say that $(X_N)_{N\in\mathbb{N}}$ converges mod-$\phi$ with parameters $w_N$, normalization function $\phi$ and limiting function $\psi$.
\end{definition}

This definition comes with several notable consequences. We state two of them here. From the definition, it follows that 
$\phi(0)=0$ and $\psi(0)=1$. If $\phi''(0)>0$, fix $u\in\R$ and put
$t_N:=u/\sqrt{w_N\phi''(0)}$. A Taylor expansion at the origin gives
\begin{align*}
 &\E\left[\exp\left(it_N(X_N-w_N\phi'(0))\right)\right] \\
 &\qquad=\exp\left(w_N[\phi(it_N)-it_N\phi'(0)]\right)\frac{\E\left[\exp\left(it_N X_N\right)\right]}{\exp\left(w_N\phi(it_N)\right)}
 \longrightarrow e^{-u^2/2}.
\end{align*}
Thus, by L\'evy's continuity theorem,
\[
 \frac{X_N}{w_N}\xrightarrow[N\to\infty]{\P}\phi'(0)
 \qquad\text{and}\qquad
 \frac{X_N-w_N\phi'(0)}{\sqrt{w_N\phi''(0)}}
 \xrightarrow[N\to\infty]{\mathrm d}\mathcal{N}(0,1).
\]
Moreover, for any positive sequence $(r_N)_{N\in\N}$ with
$\sqrt{w_N}\ll r_N\ll w_N$, $(X_N-w_N\phi'(0))/r_N$ satisfies a moderate deviation principle with
speed $r_N^2/w_N$ and rate function $x\mapsto x^2/(2\phi''(0))$. Indeed, this
follows from the G\"artner--Ellis theorem since, for every $\lambda\in\R$,
\[
 \frac{w_N}{r_N^2}\log\E\left[
 \exp\left(\lambda\frac{r_N}{w_N}
 (X_N-w_N\phi'(0))\right)\right]
 \longrightarrow \frac{\lambda^2}{2}\phi''(0),
\]
The characteristic-function argument is also recorded in
\cite[Remark~3.10]{feray2018mod}. For the G\"artner--Ellis theorem, see
\cite[Theorem~2.3.6]{DemboZeitouni}. See also
\cite[Chapter~9.2]{KabluchkoSteigenbergerThaele}.

The main theorem of this work describes the mod-$\phi$ convergence of the rescaled random variables $X_N/s_N$ in this setting under fairly natural assumptions on the arrays $a_{N,j}$ and $b_{N,j}$. For the proof of the theorem, see Section \ref{sec: proofs}. We shall write $C^2[0,1]$ for the space of twice continuously differentiable functions on $[0,1]$. Moreover, in our setting we will automatically have $\beta_{-} = -\infty$ and $\beta_{+} = +\infty$, so that these parameters need not be mentioned in the subsequent discussion. Throughout, $\log$ denotes the principal branch of the logarithm on $\C\setminus(-\infty,0]$. Moreover, we use the convention $\frac{e^{zu}-1}{zu}=1$ at $u=0$, which holds by analytic continuation.

\begin{theorem}[Mod-$\phi$ convergence for cyclotomic generating functions]\label{thm:modphi-mixed}
	Let $(X_N)_{N\in\N}$ be a sequence of random variables and consider the rescaled moment generating functions $\E\!\big[e^{zX_N/s_N}\big]$, which are of the form \eqref{eq:DefKN},
 where $s_N\to\infty$ and $b_{N,j} \ge a_{N,j}$ for all $j=1, \ldots, N$ and all $N\in\N$.  Suppose there exists a 
	non-decreasing function $f\in C^{2}[0,1]$ with $f(1)>0$ such that $b_{N,j}/s_N=f(j/N)$ for all $j=1,\ldots,N$.  Further, assume  $\delta_{N,j}:=a_{N,j}/s_N$ satisfies
	\[
	\max_{1\le j\le N} \delta_{N,j} \to 0, \qquad 
	\sum_{j=1}^{N}\delta_{N,j} \to \Sigma \in [0,\infty),
	\]
	as $N\to\infty$. Then the sequence of random variables ${X_N}/{s_N}$ converges mod-$\phi$ with parameter sequence $w_N=N$, domain $\widetilde{\mathscr{D}} = \{z\in \C: \Im z \in (-\pi/M, \pi/M)\}\cup \{z\in\C: \Re z < 0\}$ with $M := f(1)$, and normalization function and limiting function given by
	\begin{align*}
	\phi(z) &= \int_{0}^{1} \log\Bigl(\frac{e^{z f(t)}-1}{z f(t)}\Bigr)\,dt, \\
	\psi(z) &= \exp\left(
		\frac12\,\log\Bigl(\frac{e^{z f(1)}-1}{z f(1)}\Bigr)
		- \frac12\,\log\Bigl(\frac{e^{z f(0)}-1}{z f(0)}\Bigr)
		- \frac{z}{2}\,\Sigma
	\right).
	\end{align*}
If, in addition, we have $f(0) = 0$, then the same result holds on
$$
\mathscr{D} := \bigl\{z\in \C: \frac{e^{z f(t)}-1}{z f(t)}\not\in (-\infty, 0] \text{ for all } t\in [0,1]\bigr\}.
$$
\end{theorem}

\begin{remark}\label{rmk:interpretation-eta-psi}
	Theorem \ref{thm:modphi-mixed} reveals a separation of scales in the asymptotic behavior of $X_N/s_N$. The normalization function $\phi$ is determined by the global shape of the profile $f$ via its integral over $[0,1]$. In contrast, the limiting function $\psi$ depends on $f$ only through the boundary values $f(0)$ and $f(1)$. Furthermore, the numerator parameters $a_{N,j}$ act as a perturbation in the sense that they do not affect the exponential rate $\phi$, but influence the limit $\psi$ solely through the aggregated shift constant $\Sigma$.
\end{remark}

\begin{remark}\label{rmk: connectedness of Df}
	Notice that the condition $f(0) = 0$ is needed for $\mathscr{D}$ to be star-shaped (which is used in the proof of Lemma \ref{lem: properties of Df}). We refer to Figure \ref{fig:plot1} for an example with a function that violates this property. We further remark that $\widetilde{\mathscr{D}}\subset \mathscr{D}$, see Lemma \ref{lem: properties of Df}. 
\end{remark}

\begin{figure}[ht]
\centering
\includegraphics[width=0.5\textwidth]{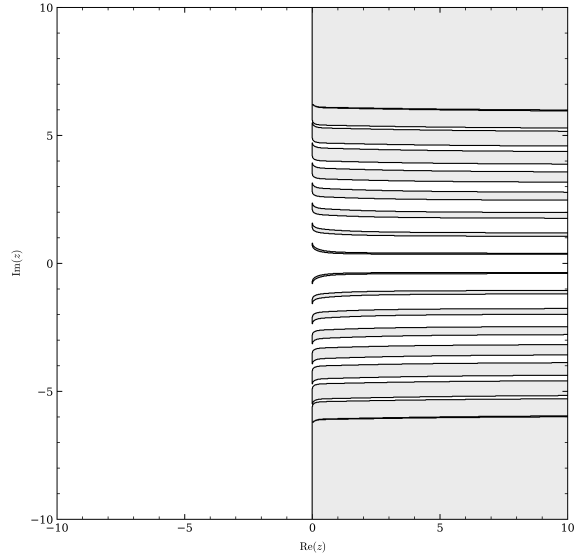}
\caption{The white region is the domain $\mathscr{D}$ for the function $f(t) = t^2 + 8$.}
\label{fig:plot1}
\end{figure}

\section{Application to combinatorial statistics}\label{sec: applications}

In this section we apply the main theorem to specific combinatorial statistics. All proofs can be found in Section \ref{sec:proofs_app}.

\subsection{Sums of i.i.d.\ discrete uniform random variables}

We begin our applications of Theorem \ref{thm:modphi-mixed} with a benchmark case: the sum of i.i.d.\ discrete uniform random variables. While this model can be studied through classical Fourier-analytic means, it serves as a helpful illustration of how the profile function $f$ and the shift $\Sigma$ are determined.

For each $N\in\N$, let $Y_{N,1},\dots,Y_{N,N}$ be i.i.d.\ with
$\P(Y_{N,j}=k)=1/N$ for $k=0,\dots,N-1$, and set
$X_N:=\sum_{j=1}^{N}Y_{N,j}$. For each summand, we have
$$
\E[e^{z\,Y_{N,j}/N}]
=\frac1N\sum_{k=0}^{N-1}e^{zk/N}
=\frac{1-e^{z}}{N\,(1-e^{z/N})},
$$
and we see that roots are the $N$-th roots of unity excluding $1$. For the rescaled summands, we have
$$
\E\big[e^{z\,X_N/N}\big]
=\Bigl(\frac{1}{N}\Bigr)^{N}\big(\frac{1-e^{z}}{1-e^{z/N}}\big)^{N}
=\Bigl(\prod_{j=1}^{N}\frac{1}{N}\Bigr)\Big(\prod_{j=1}^{N}\frac{1-e^{z}}{1-e^{z/N}}\Big).
$$
The following statement is an immediate consequence of Theorem~\ref{thm:modphi-mixed}, it could likewise be inferred from the general mod-$\phi$ convergence results established in~\cite{feray2018mod}.

\begin{theorem}[Mod-$\phi$ convergence for sums of discrete uniforms]\label{thm: mod_phi_uniform_sum}
Let $X_N$ be as above. Then $X_N/N$ converges mod-$\phi$ with parameter $w_N=N$, and normalization function and limiting function given by
\[
\phi(z):=\log\Bigl(\frac{e^{z}-1}{z}\Bigr),\qquad \psi(z):=e^{-z/2},
\]
on the domain $\mathscr{D} = \{z\in \C: \Re z < 0 \text{ or } \Im z \in (-\pi, \pi)\}$.
\end{theorem}

\subsection{Number of inversions in a random permutation}

An inversion of a permutation $\pi$ on the set $\{1,\ldots,N\}$ is a pair $(i,j)\in\{1,\ldots,N\}^2$ with $1\le i < j \le N$ and $\pi(i) > \pi(j)$. Let $X_N$ be the number of inversions in a uniform random permutation of the set $\{1,\ldots,N\}$. The moment generating function for the number of inversions is given by
\[ 
\E\big[e^{z X_N}\big] = \frac{1}{N!}\prod_{j=1}^N \frac{1-e^{zj}}{1-e^z},
\]
see, for example, \cite{heerten2024probabilistic}.  In other words, $X_N=Y_1+\ldots+Y_N$ with independent random variables $Y_j$, each uniformly distributed on $\{0,\ldots,j-1\}$, $j=1,\ldots,N$. In particular, for the rescaled variable $X_N/N$ we have
\[ 
\E\big[e^{z X_N/N}\big] = \frac{1}{N!}\prod_{j=1}^N \frac{1-e^{zj/N}}{1-e^{z/N}}.
\]
The following is a consequence of the main theorem.

\begin{theorem}[Mod-$\phi$ convergence for number of inversions]\label{thm: mod_phi_convergence_inversions}
	For each $N\in\N$ let $X_N$ be the number of inversions of a uniform random permutation of size $N$. Then the sequence of random variables $X_N/N$ converges mod-$\phi$ to the limiting function
	\[
	\psi(z) := \sqrt{\frac{e^z-1}{z}}e^{-z/2}
	\]
	with parameters $w_N=N$ and normalization function
	\[
	\phi(z) := \int_0^1 \log\Bigl(\frac{e^{tz}-1}{{tz}}\Bigr)dt
	\]
	on the domain $\mathscr{D} = \{z\in \C: \frac{e^{tz}-1}{{tz}}\not\in (-\infty, 0] \text{ for all } t\in [0,1]\}$.
\end{theorem}

The next lemma provides a closed form representation for the function $\phi$ of Theorem \ref{thm: mod_phi_convergence_inversions} on $\mathscr{D}\setminus [0,\infty)$ in terms of Euler's dilogarithm $\operatorname{Li}_2(z)$, which is defined by the power series
$$
\operatorname{Li}_2(z) := \sum_{k=1}^{\infty} \frac{z^k}{k^2}, \qquad \text{ for } |z| < 1,
$$
and by analytic continuation on $\mathbb{C} \setminus [1,\infty)$.

\begin{lemma}\label{lem: integral identity}
For $z \in \mathscr{D}$, with $\mathscr{D}$ as in Theorem \ref{thm: mod_phi_convergence_inversions}, we have 
\begin{align*}
	\int_0^1 \log\Bigl(\frac{e^{tz}-1}{{tz}}\Bigr)dt = -\frac{1}{z}\Bigl(\operatorname{Li}_2(e^z)-\frac{\pi^2}{6}\Bigr)-\log(-z)+1.
\end{align*}
\end{lemma}

Theorem \ref{thm: mod_phi_convergence_inversions} deals with the number of inversions of a random permutation in the classical Coxeter group $A_{N-1}$. The Coxeter group $B_N$ corresponds similarly to inversions of so-called signed permutations. Let now $X_N$ denote the number of inversions of a uniform random signed permutation of $B_N$. Then, the moment generating function of $X_N/N$ is given by
\[
\E[e^{zX_N/N}] = \frac{1}{ 2^NN!} \prod_{j=1}^N \frac{1-e^{2zj/N}}{1-e^{z/N}},
\]
see \cite[Chapter 7]{bjorner2005combinatorics} or \cite[Chapter 2]{lehrer2009unitary} for more details. For this Coxeter group, we have the following result.

\begin{theorem}[Mod-\(\phi\) convergence for \(B_N\)]\label{thm:modphi-colored}
Let $X_N$ denote the number of inversions of a random signed permutation of $B_N$. Then \(X_N/N\) converges mod-\(\phi\) with parameter \(N\), domain
\[
	\mathscr{D}=\{z\in\C:\ \tfrac{e^{2zt}-1}{2zt}\notin(-\infty,0] \text{ for all } t\in[0,1]\},
\]
normalization function
\[
\phi(z)=\int_0^1 \log\Bigl(\frac{1-e^{2zt}}{-2zt}\Bigr)\,dt,
\]
and limiting function
\[
\psi(z)=\exp\Bigl(\frac12\log\Bigl(\frac{1-e^{2z}}{-2z}\Bigr)-\frac{z}{2}\Bigr)
=\sqrt{\frac{e^{2 z}-1}{2 z}}\;e^{-z/2}.
\]
\end{theorem}

\begin{remark}
It is possible to derive an alternative representation for $\phi$. Namely, for $z\in\mathscr{D}\setminus[0, \infty)$,
$$
\phi(z)
= -\frac{1}{2z}\Bigl(\operatorname{Li}_2(e^{2z})-\frac{\pi^2}{6}\Bigr)
- \log(-2z) + 1.
$$
\end{remark}

As mentioned above, Mod-$\phi$ convergence provides a unified route to central and local limit theorems as well as precise moderate and large deviation asymptotics. For the inversion statistic, asymptotic normality and refined coefficient asymptotics have been obtained by saddle-point methods in \cite{Bender1973CLLT,LouchardProdinger2003Inversions}, see also \cite{Margolius2001PermutationsInversions}. Large deviation principles for atypical inversion densities can moreover be deduced from \cite{BorgaDasMukherjeeWinkler2024LDP}, where the case of inversions is treated explicitly. For the signed permutation groups $B_N$, central and local limit theorems are proved in \cite{KahleStump2020CoxeterInversions}.

\subsection{Descending plane partitions}

Our next application demonstrates the flexibility of Theorem \ref{thm:modphi-mixed} by considering a model where the profile function $f(t)$ is non-linear and the shift parameters $a_{N,j}$ are non-constant. Precisely, we consider a sequence of random variables $X_N$ whose rescaled moment generating functions are of the form
\begin{equation}\label{eq:mgfDPP}
\E[e^{zX_N/N^2}] = \frac{1}{N!} \prod_{j=1}^N \frac{1-e^{zj^2/N^2}}{1-e^{zj/N^2}}.
\end{equation}
This is a scaled form of the generating polynomial of the number of so-called descending plane partitions of size $N$. Intuitively, descending plane partitions can be visualized as two-dimensional arrays of positive integers where the entries strictly decrease along each row. Their primary interest in the literature stems from the celebrated fact that they are equinumerous with alternating sign matrices. For background material we refer to \cite[Section 4.3.4]{heerten2024probabilistic} and the references cited therein.

\begin{theorem}[Mod-$\phi$ convergence for descending plane partitions]\label{thm: mod_phi_j_squared}
	For each $N\in\N$ let $X_N$ be a random variable with rescaled moment generating function of the form \eqref{eq:mgfDPP}. Then the sequence of random variables $X_N/N^2$ converges mod-$\phi$ with parameter $N$, and normalization function and limiting function given by
	\begin{align*}
	\phi(z) &:= \int_0^1\log\Bigl(\frac{1-e^{zt^2}}{-zt^2}\Bigr) dt, \\
	\psi(z) &:= \exp\Bigl( \frac{1}{2}\log\Bigl(\frac{1-e^{z}}{-z}\Bigr) - \frac{z}{4} \Bigr) = \sqrt{\frac{e^z-1}{z}}\, e^{-z/4}.
	\end{align*}
	on $\mathscr{D}$ as in Theorem \ref{thm:modphi-mixed}. 
\end{theorem}

\subsection{Inversions in random multiset permutations}

Fix $N\in\N$ and a vector  $m^{(N)}=(m_1^{(N)},\dots,m_{r(N)}^{(N)})$ of natural numbers with $\sum_{i=1}^{r(N)} m_i^{(N)}=n_N \in\N$. Let $\pi_N$ be a uniformly random permutation of the multiset having $m_i^{(N)}$ copies of the letter $i$ for each $i$, and let $X_N$ denote its number of inversions. By MacMahon's $q$-multinomial formula \cite[Equations (2) and (3)]{novick2009bijective},
\[
\sum_{\pi} q^{\mathrm{inv}(\pi)}=\binom{n_N}{m_1^{(N)},\dots,m_{r(N)}^{(N)}}_{q}
=\frac{\prod_{j=1}^{n_N}(1-q^{j})}{\prod_{i=1}^{r(N)}\prod_{j=1}^{m_i^{(N)}}(1-q^{j})}
\]
is the polynomial whose coefficient of $q^k$ equals the number of distinct multiset permutations with exactly $k$ inversions. Let $T = \frac{n_N!}{\prod_{i=1}^{r(N)} m_i^{(N)}!}$ be the total number of multiset permutations (which corresponds to the value at $q\to 1$) and let $N_k$ be the number of permutations with $k$ inversions. Then,
$$
\E[q^{X_N}] = \frac{1}{T}\sum_{k\ge 0} N_k q^k.
$$
Hence, setting $q=e^{z/n_N}$ yields the rescaled moment generating function
\[
\E\big[e^{z X_N/n_N}\big] = \Big(\prod_{j=1}^{n_N}\frac{a_{N,j}}{j}\Big)\Big(\prod_{j=1}^{n_N}\frac{1-e^{z\,j/n_N}}{1-e^{z\,a_{N,j}/n_N}}\Big), \] where the sequence of numerator parameters is $b_{N,j}=j$ for $j=1,\ldots,n_N$, and the denominator parameters $a_{N,j}$ are given by the concatenated sequence \[ (1,2,\dots,m_1^{(N)}, \dots, 1,2,\dots,m_{r(N)}^{(N)}). \] Note that $\sum_{i=1}^{r(N)} m_i^{(N)} = n_N$, so there are exactly $n_N$ terms in both the numerator and denominator products, matching the form of \eqref{eq:DefKN}.

\begin{theorem}[Mod-$\phi$ convergence for multiset inversions]\label{thm:modphi-multiset-inversions}
Assume that 
$$
M_N:=\max_{1\le i\le r(N)} m_i^{(N)} =  o(n_N),
$$
and set $s_N=n_N$. Suppose the limit
\[
\Sigma\ :=\ \lim_{N\to\infty}\frac{1}{n_N}\sum_{i=1}^{r(N)}\frac{m_i^{(N)}\bigl(m_i^{(N)}+1\bigr)}{2}
\]
exists in $\R$. Then $X_N/n_N$ converges mod-$\phi$ with parameter $n_N$, normalization function
\[
\phi(z)\ :=\ \int_0^1 \log\Bigl(\frac{e^{tz}-1}{tz}\Bigr)\,dt,
\]
domain $\mathscr{D}=\{z\in\C:\ \frac{e^{tz}-1}{tz}\notin(-\infty,0]\ \text{for all }t\in[0,1]\}$, and limiting function
\[
\psi(z) = \exp\Bigl(\frac12\log\Bigl(\frac{e^{z}-1}{z}\Bigr)-\frac{z}{2}\,\Sigma\Bigr)
\ =\ \sqrt{\frac{e^{z}-1}{z}}\;e^{-\,\frac{z}{2}\Sigma}.
\]
\end{theorem}

\begin{remark}
	If $m_i^{(N)}\equiv 1$ for all $i$, then $\Sigma=1$ and Theorem \ref{thm:modphi-multiset-inversions} reduces to Theorem \ref{thm: mod_phi_convergence_inversions} for ordinary permutations. More generally, this result shows that as long as the largest multiplicity $M_N$ is small compared to the total size $n_N$, the global fluctuations are governed by the same profile $f(t)=t$ as in the classical case, with the specific multiset structure only appearing as a shift in the limiting function $\psi$.
\end{remark}

Classical limit theorems for the inversion number in random multiset permutations (random words with fixed content) are well-established. In particular, \cite{CongerViswanath2007} proves a quantitative normal approximation using Stein's method, and \cite{CanfieldJansonZeilberger2011} establishes asymptotic normality for the inversion count and discusses a local limit theorem. Theorem \ref{thm:modphi-multiset-inversions} complements these results by providing an alternative route to these limit theorems. In addition, it implies a  description of precise moderate and large deviation estimates.

\section{Proof of Theorem \ref{thm:modphi-mixed}}\label{sec: proofs}

In this section we will prove the main result of this work. Our first goal is to find  conditions on $a_{N,j}, b_{N,j}, s_N$ which ensure that $\lim_{N\to\infty} \exp(-N \phi(z)) K_N(z) = \psi(z)$ locally uniformly on some set $\mathscr{D} \subset \C$ for suitable functions $\phi$ and $\psi$.

\begin{lemma}[General Case Limit]\label{lem: general asymptotic framework}
Let $(s_N)_{N\in\N}$ be a sequence of positive integers. For each $N\in\N$, let $(a_{N,j})_{j=1}^N$ and $(b_{N,j})_{j=1}^N$ be sequences of positive integers. Define $h(w, z) := \log(\frac{1-e^{zw}}{-zw})$ for $w>0$, with $h(0, z) := 0$.
Recall the sequence
\[
K_N(z) := \Big( \prod_{j=1}^N \frac{a_{N,j}}{b_{N,j}} \Big) \Big(\prod_{j=1}^N \frac{1-e^{z b_{N,j}/s_N}}{1-e^{z a_{N,j}/s_N}}\Big).
\]
Define the sums $S_N^a(z) := \sum_{j=1}^N h(a_{N,j}/s_N, z)$ and $S_N^b(z) := \sum_{j=1}^N h(b_{N,j}/s_N, z)$. Assume there exist functions $I_a, I_b, C_a, C_b$ such that the following asymptotic expansions hold locally uniformly for $z$ in some domain $\mathscr{D} \subset \mathbb{C}$:
\begin{align*}
	S_N^a(z) &= N I_a(z) + C_a(z) + o(1),\qquad\text{and}\qquad S_N^b(z) = N I_b(z) + C_b(z) + o(1),
\end{align*}
as $N\to\infty$. Let $\phi(z) := I_b(z) - I_a(z)$. Then,
\begin{equation}\label{eq:21-10-25a}
	\lim_{N\to\infty} \exp\big(-N \phi(z) \big) K_N(z) = \exp\big( C_b(z) - C_a(z) \big)
\end{equation}
locally uniformly on $\mathscr{D}$.
\end{lemma}

\begin{proof}
We work with the principal branch of the logarithm. For $z \in \mathscr{D}$ such that the denominator terms do not vanish, the definition of $h(w,z)$ gives
\[
1-e^{zw}=-zw\exp(h(w,z)).
\]
Substituting this identity directly into the product defining $K_N(z)$, we obtain
\begin{align*}
	K_N(z) &= \Big( \prod_{j=1}^N \frac{a_{N,j}}{b_{N,j}} \Big) \Big(\prod_{j=1}^N \frac{-z(b_{N,j}/s_N)\exp(h(b_{N,j}/s_N,z))}{-z(a_{N,j}/s_N)\exp(h(a_{N,j}/s_N,z))}\Big) \\
	&= \exp\Big(S_N^b(z) - S_N^a(z)\Big).
\end{align*}
Using the assumed asymptotic expansions for $S_N^a$ and $S_N^b$, we obtain
\begin{align*}
	K_N(z) &= \exp\Big( \big(N I_b(z) + C_b(z) + o(1)\big) - \big(N I_a(z) + C_a(z) + o(1)\big) \Big) \\
	&= \exp\Big( N\phi(z) + C_b(z) - C_a(z) + o(1) \Big).
\end{align*}
Multiplying by $\exp(-N\phi(z))$ and taking the limit $N\to\infty$ yields the asserted result.
\end{proof}

\begin{remark}
Note that the right-hand side of \eqref{eq:21-10-25a} is always positive for real-valued $z$. In particular, it does not vanish on $\mathscr{D}\cap\R$, as required in Definition~\ref{def: mod_phi_convergence}.
\end{remark}

We will need in our next argument the trapezoid rule for the integral of a function $f$ on an interval $[a,b]$ that can be found in \cite{cerone2019trapezoidal} or \cite[Section 3, Corollary 1]{dragomir1999some}.

\begin{proposition}[Trapezoid rule]\label{prop: trapezoid rule}
	Let $a<b$ and $f\in C^2[a,b]$ be a twice continuously differentiable function. Then, for any $n\in\N$,
	\[
	\int_a^b f(x)\,dx = A_{n}(f) + R_{n}(f),
	\]
	where
	\[
	A_{n}(f) = \frac{b-a}{n}\Big( \frac{f(a)+f(b)}{2} + \sum_{i=1}^{n-1} f\bigl(a + \frac{b-a}{n}i\bigr) \Big)
	\]
	and the remainder satisfies $|R_{n}(f)| \le \frac{(b-a)^3 \|f''\|_{\infty}}{12n^2}$.
\end{proposition}

We now proceed to derive the asymptotic expansions for the sums $S_N^a(z)$ and $S_N^b(z)$ required by Lemma \ref{lem: general asymptotic framework}.

\begin{lemma}[Asymptotics via Trapezoid Rule]\label{lem: asymptotics E-M}
Assume $a_{N,j}/s_N = f(j/N)$ for some function $f \in C^2[0, 1]$ with $f(t) \ge 0$.
Let
$$
\mathscr{D} = \Big\{z \in \mathbb{C} \mid \frac{1-e^{zf(t)}}{-zf(t)} \notin (-\infty, 0] \text{ for all } t \in [0, 1] \text{ such that } f(t) > 0\Big\}.
$$
Then, for $z \in \mathscr{D}$, the sum $S_N^a(z) = \sum_{j=1}^N h(a_{N,j}/s_N, z)$ admits the asymptotic expansion
\[
S_N^a(z) = N \int_0^1 h(f(t), z) dt + \frac{h(f(1), z) - h(f(0), z)}{2} + O(1/N),
\]
which holds locally uniformly for $z \in \mathscr{D}$.
\end{lemma}

\begin{proof}
Let $F(t,z) := h(f(t),z)$, where we recall that $h(w,z) = \log(\frac{1-e^{zw}}{-zw})$. To apply the trapezoid rule (Proposition \ref{prop: trapezoid rule}), we must examine the regularity of $t \mapsto F(t,z)$.

Fix a compact subset $K \subset \mathscr{D}$. For any $z \in K$, the function $w \mapsto h(w,z)$ is analytic in a neighbourhood of the range of $f$. Since $f \in C^2[0,1]$, the composite function $F(t,z)$ is $C^2$ with respect to $t$ on $[0,1]$. Crucially, since $F(t,z)$ and its derivatives are continuous functions of $(t,z)$ on the compact set $[0,1] \times K$, the second partial derivative with respect to $t$ is uniformly bounded, i.e.
\[
\sup_{z \in K} \sup_{t \in [0,1]} \left| \frac{\partial^2}{\partial t^2} F(t,z) \right| \le C_K < \infty.
\]
Proposition \ref{prop: trapezoid rule} implies that for any fixed $z \in K$,
\[
\int_0^1 F(t, z) dt = \frac{1}{N}\Big(\frac{F(0, z) + F(1, z)}{2} + \sum_{j=1}^{N-1} F(j/N, z)\Big) + R_N(z),
\]
where $|R_N(z)| \le \frac{C_K}{12 N^2}$. Multiplying by $N$ and rearranging the terms yields
\[
\sum_{j=1}^{N-1} F(j/N, z) = N \int_0^1 F(t, z) dt - \frac{F(0, z) + F(1, z)}{2} + O(1/N),
\]
where the error term is uniform in $z \in K$. Finally, adding the last term $F(1,z) = F(N/N, z)$ to both sides, we obtain
\begin{align*}
	\sum_{j=1}^{N} F(j/N, z) &= N \int_0^1 F(t, z) dt - \frac{F(0, z) + F(1, z)}{2} + F(1,z) + O(1/N) \\
	&= N \int_0^1 F(t, z) dt + \frac{F(1, z) - F(0, z)}{2} + O(1/N).
\end{align*}
This completes the proof.
\end{proof}

Next, we study geometric properties of the domains $\mathscr{D}$ and $\widetilde{\mathscr{D}}$ introduced in Lemma \ref{lem: asymptotics E-M} and Theorem \ref{thm:modphi-mixed}.

\begin{lemma}[Properties of the domain]\label{lem: properties of Df}
Let $f \in C[0, 1]$ be a non-negative function. Then the set $\mathscr{D}$ is open.
If, in addition, $f$ is monotonically increasing and $f(1)>0$, then $\widetilde{\mathscr{D}}$ is open and connected. If, furthermore, $f(0) = 0$, then $\mathscr{D}$ is star-shaped and hence connected.
\end{lemma}
\begin{proof}
Let $g(w) := \frac{1-e^{w}}{-w}$ for $w \ne 0$ and $g(0) := 1$. This function is entire. Recall that the set $\mathscr{D}$ is defined as
\[
\mathscr{D} = \{z \in \mathbb{C} \mid g(zf(t)) \notin (-\infty, 0] \text{ for all } t \in [0, 1] \text{ with } f(t)>0\}.
\]
To see that $\mathscr{D}$ is open, fix $z_0\in \mathscr{D}$ and consider the continuous map $H(z,t):=g(zf(t))$ on $\C\times[0,1]$. Because $g(z_0f(t))$ avoids the closed set $(-\infty,0]$ for every $t$, the compactness of $[0,1]$ implies that the distance
\[
\delta:=\inf_{t\in[0,1]}\,d\bigl(g(z_0f(t)),(-\infty,0]\bigr)
\]
is strictly positive. The uniform continuity of $H$ on some compact neighbourhood of $(z_0,[0,1])$ provides an $\varepsilon>0$ such that $|z-z_0|<\varepsilon$ implies $|H(z,t)-H(z_0,t)|<\delta/2$ for all $t$ and $z$ in that neighbourhood. Consequently, $g(zf(t))$ remains outside $(-\infty,0]$. Therefore $B(z_0,\varepsilon)\subset \mathscr{D}$, which shows that $\mathscr{D}$ is open.

Next, we establish connectedness, considering first the case where $f(0)$ is not necessarily zero. Assume that $f$ is monotonically increasing and set $M:=f(1)>0$. Clearly, if $\Re z<0$ then $g\bigl(zf(t)\bigr)\notin(-\infty,0]$ for all $t\in[0,1]$, so $\{z\in\C:\Re z<0\}\subset \mathscr{D}$. It remains to show that $g(w)\notin(-\infty,0]$ whenever $|\,\Im w\,|<\pi$. Using the representation $g(w)=\int_{0}^{1}e^{sw}\,ds$, we compute for $w=x+iy$,
\[
\Im g(x+iy)=\int_{0}^{1}e^{sx}\sin(sy)\,ds.
\]
Since $e^{sx}$ is positive, the sign of the integrand depends on the sign of $\sin(sy)$. If $y\in (0,\pi)$, then $sy\in (0, \pi)$ for $s\in(0,1)$, and thus $\sin(sy)$ is positive, making the integral strictly positive. Similarly, if $y\in (-\pi, 0)$, the integral is strictly negative. Hence, $g(w)\notin(-\infty,0]$ whenever $|\,\Im w\,|<\pi$. Since $0\le f(t)\le M$ for $t\in[0,1]$, it follows that $|\,\Im z\,|<\pi/M$ implies $|\,\Im(zf(t))\,|<\pi$ for all $t$, and thus $z\in\mathscr{D}$. This shows that the set
\[
\widetilde{\mathscr{D}}=\{z\in\C:\Re z<0\}\cup\{z\in\C:|\,\Im z\,|<\pi/M\}
\]
is contained in $\mathscr{D}$ and is connected.

Finally, assume $f(0) = 0$. Since $f$ is continuous and increasing, we have $f([0,1])=[0,M]$. Let $S:=\{w\in\C: g(w)\in(-\infty,0]\}$. Fix $z\in \mathscr{D}$. Then $g\bigl(zf(t)\bigr)\notin(-\infty,0]$ for all $t\in(0,1]$ with $f(t)>0$, which implies $g(zy)\notin(-\infty,0]$ for all $y\in(0,M]$. Therefore, the line segment
\[
\Gamma_z:=\{zy:0<y\le M\}
\]
is disjoint from $S$. It follows that every shorter segment
\[
\Gamma_{\lambda z}:=\{\lambda zy:0<y\le M\},\qquad 0<\lambda\le 1,
\]
is also disjoint from $S$. Hence, $\lambda z\in \mathscr{D}$ for all $\lambda\in(0,1]$, meaning the full line segment $[0,z]$ lies inside of $\mathscr{D}$. Therefore, $\mathscr{D}$ is star-shaped with respect to the origin and, in particular, connected.
\end{proof}

In a next step we study analyticity properties of the normalization function $\phi$.

\begin{lemma}[Properties of the limiting integral]\label{lem: properties of eta_f}
Let $f \in C[0,1]$ be a non-negative function that is not identically zero. The function $\phi$ defined by
\[
\phi(z) := \int_0^1 \log\Bigl(\frac{e^{zf(t)}-1}{zf(t)}\Bigr)\,dt
\]
is analytic on the domain 
$$
\mathscr{D} = \{z \in \mathbb{C} \mid \frac{e^{zf(t)}-1}{zf(t)} \notin (-\infty, 0] \text{ for all } t \in [0, 1] \text{ with } f(t) > 0\}.
$$
Its restriction to the real axis is real-valued and strictly convex.
\end{lemma}
\begin{proof}
Let $k(w) := \log\bigl(\frac{e^w-1}{w}\bigr)$. This function is analytic on the set $\mathbb{C} \setminus \{w : \frac{e^w-1}{w} \in (-\infty, 0]\}$ and takes real values for real $w$. Because $f$ is continuous and non-negative, the map $(z,t) \mapsto k(zf(t))$ is jointly continuous on $\mathscr{D} \times [0,1]$ and analytic in $z$ for each fixed $t$. Therefore, the parameter integral
\[
\phi(z) = \int_0^1 k(zf(t))\,dt
\]
is analytic on $\mathscr{D}$. Specifically, in a neighborhood of $z=0$, we can expand $k(w) = \sum_{n=1}^{\infty} c_n w^{n}$. Since $f$ is bounded, the series for $k(z f(t))$ converges uniformly for $t\in[0,1]$, justifying term-by-term integration:
\[
\phi(z) = \sum_{n=1}^{\infty} c_n z^{n} \int_{0}^{1} (f(t))^{n}\,dt.
\]
Thus, $\phi$ is analytic at $z=0$ with $\phi(0)=0$.

For real $x \in \mathscr{D} \cap \mathbb{R}$, the integrand is real, so $\phi(x) \in \mathbb{R}$. By the  Leibniz integral rule, we may differentiate under the integral sign:
\[
\phi''(x) = \int_0^1 (f(t))^2 k''\bigl(x f(t)\bigr)\,dt.
\]
Note that $k(u) = \log \mathbb{E}[e^{uU}]$ where $U$ is uniformly distributed on $[0,1]$. Thus, $k''(u) = \text{Var}_u(U) > 0$ for all $u \in \mathbb{R}$, where $\text{Var}_u(U)$ refers to the variance of the random variable $U$ under the exponentially tilted measure with parameter $u$. Since $f$ is not identically zero and continuous, $f(t) > 0$ on a set of positive measure. The integrand is non-negative everywhere and strictly positive on a set of positive measure, implying $\phi''(x) > 0$. Consequently, $\phi$ is strictly convex on $\mathbb{R}$.
\end{proof}

We now verify that, under the assumptions of Theorem \ref{thm:modphi-mixed}, the expansion for $S_N^a$ required by Lemma \ref{lem: general asymptotic framework} holds.

\begin{lemma}[Asymptotics via Taylor Expansion]\label{lem: asymptotics Taylor}
Let $s_N \in \N$ be a sequence of natural numbers and, for each $N\in\N$, let $(a_{N,j})_{j=1}^N$ be a sequence of natural numbers. Let $h(w, z) := \log\bigl(\frac{1-e^{zw}}{-zw}\bigr)$ for $w \neq 0$ and $h(0, z):=0$. For each $N \in \N$, set $\delta_{N,j} := a_{N,j}/s_N$. Assume that $\max_{1\le j \le N} \delta_{N,j} \to 0$ as $N\to\infty$ and that the limit
\[
\Sigma := \lim_{N\to\infty} \sum_{j=1}^N \delta_{N,j}
\]
exists in $\R$. Then, for every compact set $K\subset\C$, the sum $S_N^a(z) = \sum_{j=1}^N h(a_{N,j}/s_N, z)$ admits the asymptotic expansion
\[
S_N^a(z) = \frac{z}{2}\,\Sigma + o(1),
\]
uniformly for $z\in K$.
\end{lemma}
\begin{proof}
Fix a compact set $K\subset\C$. We first determine the local behavior of $h(w,z)$ near $w=0$. Expanding the exponential term for small $|w|$, we have
\[
\frac{1-e^{zw}}{-zw} = \frac{1 - (1 + zw + \frac{1}{2}z^2w^2 + O(w^3))}{-zw} = 1 + \frac{z}{2}w + O(w^2).
\]
Taking the logarithm yields
\[
h(w,z) = \log\Bigl(1 + \frac{z}{2}w + O(w^2)\Bigr) = \frac{z}{2}w + O(w^2).
\]
Thus, there exist constants $C_K>0$ and $\varepsilon>0$ such that
\[
\Bigl|h(w,z) - \frac{z}{2}w\Bigr| \le C_K |w|^2 \qquad\text{for all }z\in K\text{ and }|w|\le\varepsilon.
\]
Since $\max_j \delta_{N,j} \to 0$, we may assume $\delta_{N,j} \le \varepsilon$ for all $j$ provided $N$ is sufficiently large. We define the remainder terms $R_{N,j}(z) := h(\delta_{N,j},z) - \frac{z}{2}\delta_{N,j}$. The sum can then be written as
\[
S_N^a(z) = \sum_{j=1}^N \frac{z}{2}\delta_{N,j} + \sum_{j=1}^N R_{N,j}(z).
\]
The first term converges to $\frac{z}{2}\Sigma$ by assumption. For the remainder sum, we apply the quadratic bound
\[
\Bigl|\sum_{j=1}^N R_{N,j}(z)\Bigr| \le \sum_{j=1}^N C_K \delta_{N,j}^2 \le C_K \Bigl(\max_{1\le j\le N}\delta_{N,j}\Bigr)\sum_{j=1}^N \delta_{N,j}.
\]
The right-hand side tends to $0$ as $N\to\infty$ because the maximum tends to $0$ while the sum converges to $\Sigma$. Thus, the remainder vanishes uniformly for $z\in K$, proving the claim.
\end{proof}

We are now in a position to complete the proof of Theorem \ref{thm:modphi-mixed}. Recall that $h(w,z)=\log(\frac{1-e^{zw}}{-zw})$ with $h(0,z)=0$.

\begin{proofof}{Theorem {\ref{thm:modphi-mixed}}}
First, we determine the asymptotic behavior of the numerator term $S_N^{b}(z)$ using Lemma~\ref{lem: asymptotics E-M} with the profile function $f$. This yields
\begin{align*}
S_N^{b}(z) &= N \int_0^1 h(f(t),z)\,dt + \frac{h\bigl(f(1),z\bigr) - h\bigl(f(0),z\bigr)}{2} + o(1)\\
&= N\phi(z) + \frac{h\bigl(f(1),z\bigr) - h\bigl(f(0),z\bigr)}{2} + o(1),
\end{align*}
locally uniformly for $z \in \mathscr{D}$.
Next, for the denominator term $S_N^{a}(z)$, Lemma~\ref{lem: asymptotics Taylor} provides the expansion
\[
S_N^{a}(z) = \frac{z}{2}\,\Sigma + o(1).
\]
Substituting these two expansions into the general framework of Lemma~\ref{lem: general asymptotic framework}, we obtain
\[
\exp\bigl(-N\phi(z)\bigr)\,K_N(z) \longrightarrow \exp\Bigl(\frac{h\bigl(f(1),z\bigr) - h\bigl(f(0),z\bigr)}{2} - \frac{z}{2}\,\Sigma\Bigr) = \psi(z),
\]
as $N\to\infty$, locally uniformly for $z\in \mathscr{D}$.

To conclude, we verify the requirements (i)--(iv) of Definition~\ref{def: mod_phi_convergence}:
\begin{itemize}
	\item[(i)] The parameter sequence $w_N = N$ clearly tends to infinity.
	\item[(ii)] Regarding the domain: By Lemma~\ref{lem: properties of Df}, since $f$ is monotonic and $f(1)>0$, the set $\widetilde{\mathscr{D}}$ is an open, connected subset of $\C$ (and $\widetilde{\mathscr{D}}\subset \mathscr{D}$). In the case $f(0)=0$, the larger set $\mathscr{D}$ itself is connected.
	\item[(iii)] By Lemma~\ref{lem: properties of eta_f}, the normalization function $\phi$ is analytic on $\mathscr{D}$ and its restriction to the real axis is strictly convex.
	\item[(iv)] Finally, the limiting function $\psi$ is analytic on $\mathscr{D}$ and, as the exponential of a real-valued function for real arguments, it does not vanish on $\mathscr{D}\cap\R$.
\end{itemize}
Thus, all conditions for mod-$\phi$ convergence are satisfied and the proof is complete.
\end{proofof}

\section{Proofs of the applications}\label{sec:proofs_app}

\begin{proofof}{Theorem {\ref{thm: mod_phi_uniform_sum}}}
This follows from Theorem \ref{thm:modphi-mixed} and Remark \ref{rmk: connectedness of Df} with $s_N=N$, $a_{N,j}\equiv1$, $b_{N,j}\equiv N$, so $f(t)\equiv1$. Then $\delta_{N,j}=a_{N,j}/s_N=1/N\to0$ and we have
\[
\Sigma=\lim_{N\to\infty}\sum_{j=1}^{N}\frac1N=1.
\]
Hence,
\[
\phi(z)=\int_{0}^{1}\log\Bigl(\frac{e^{z f(t)}-1}{z f(t)}\Bigr)\,dt=\log\Bigl(\frac{e^{z}-1}{z}\Bigr),
\]
and with $f(0)=f(1)=1$ Theorem~\ref{thm:modphi-mixed} yields the limiting function
\[
\psi(z)=\exp\Bigl(\tfrac12 h(1,z)-\tfrac12 h(1,z)-\tfrac{z}{2}\Sigma\Bigr)=e^{-z/2},
\]
where $h(w,z):=\log\bigl(\frac{1-e^{zw}}{-zw}\bigr)$ as in Section~\ref{sec: proofs}. This completes the argument.
\end{proofof}

\begin{proofof}{Theorem {\ref{thm: mod_phi_convergence_inversions}}}
In this example we have $a_{N,j}=1$ and $b_{N,j}=j$ for $j=1,\ldots,N$, and $s_N=N$. Moreover, $f(t)=t$ for $t\in[0,1]$. The  hypotheses of Theorem~\ref{thm:modphi-mixed} are thus satisfied with
\[
\phi(z)=\int_{0}^{1} \log\Bigl(\frac{e^{tz}-1}{tz}\Bigr)\,dt.
\]
Here, $\delta_{N,j}=a_{N,j}/s_N=1/N\to0$ as $N\to\infty$ and we obtain
\[
\Sigma=\lim_{N\to\infty}\sum_{j=1}^{N}\frac1N=1.
\]
Together with $c_1(z)=z/2$ (implicitly used in Lemma~\ref{lem: asymptotics Taylor}), this yields the limiting function
\[
\psi(z)=\exp\Bigl(\frac12\log\Bigl(\frac{1-e^{z}}{-z}\Bigr)-\frac{z}{2}\Bigr)=\sqrt{\frac{e^{z}-1}{z}}e^{-z/2}.
\]
Hence $X_N/N$ converges mod-$\phi$ with parameter $N$, normalization function $\phi$ and limiting function $\psi$, as stated.
\end{proofof}

\begin{proofof}{Lemma {\ref{lem: integral identity}}}
We start by showing the relation
\begin{align}\label{eq:log_relation}
\log\Bigl(\frac{e^{tz}-1}{tz}\Bigr) = \log\Bigl(\frac{1-e^{tz}}{-tz}\Bigr) = \log(1-e^{tz}) - \log(-z) - \log t
\end{align}
for $z \in \mathscr{D} \setminus [0, \infty)$ and $t \in (0, 1]$. Let $x := 1-e^{tz}$ and $y := -tz$. Since $\log(-z) + \log t = \log(-zt) = \log y$, as $\operatorname{Arg}(-z) + \operatorname{Arg}(t) = \operatorname{Arg}(-z) \in (-\pi, \pi]$, the relation can be rewritten as $\log(x/y) = \log x - \log y$. This identity holds for the principal branch of the logarithm if and only if
\begin{equation} \label{eq:cond}
	\operatorname{Arg}(x) - \operatorname{Arg}(y) \in (-\pi, \pi].
\end{equation}
Let $L(z,t) := x/y = \frac{1-e^{tz}}{-tz}$. The definition of the domain $\mathscr{D}$ requires $L(z,t) \notin (-\infty, 0]$ for all $t \in [0,1]$. This means the principal argument $\operatorname{Arg}(L(z,t))$ must not be equal to $\pi$. Let $f(t) := \operatorname{Arg}(x) - \operatorname{Arg}(y)$. The principal argument of $L(z,t)$ is related to $f(t)$ by
\[
\operatorname{Arg}(L(z,t)) = f(t) + 2\pi N(t)
\]
where $N(t)$ is an integer chosen such that $\operatorname{Arg}(L(z,t)) \in (-\pi, \pi]$. Condition \eqref{eq:cond} is equivalent to $N(t)=0$. Consider the behavior of $f(t)$ as $t \to 0$ from above. We have $\lim_{t\to 0} (1-e^{tz}) = 0$ and by the power series expansion it also holds that $1-e^{tz} = -tz + O(t^2) = -tz(1+O(t))$ for sufficiently small $t$. Thus, $\lim_{t\to 0} \operatorname{Arg}(1-e^{tz}) = \lim_{t\to 0} \operatorname{Arg}(-tz)$.
Therefore, $\lim_{t\to 0} f(t) = 0$. The function $f$ is continuous for $t \in (0, 1]$, since $z \neq 0$, $t \neq 0$, and $1-e^{tz}$ is only zero if $tz=2k\pi i$, which would imply $L(z,t)$ is undefined or $L(z,t)=0$, both excluded or handled by $\mathscr{D}$.
If $f(t_0) = \pi$ for some $t_0 \in (0, 1]$, then $\operatorname{Arg}(L(z, t_0))=\pi$ with $N(t_0)=0$, which is not possible for $z \in \mathscr{D}$ as this would mean $L(z,t_0) \in (-\infty, 0]$. If $f(t_0) = -\pi$ for some $t_0 \in (0, 1]$, then $\operatorname{Arg}(L(z, t_0))=-\pi + 2\pi = \pi$ with $N(t_0)=1$, which is also not possible for $z \in \mathscr{D}$ as this would mean $L(z,t_0) \in (-\infty, 0]$. Since $f$ starts at $0$ as $t \to 0$ from above and is continuous on $(0, 1]$, and it can never reach the boundary values $\pi$ or $-\pi$ because $z \in \mathscr{D}$, it must remain strictly within the interval $(-\pi, \pi)$. Hence
\[
f(t) = \operatorname{Arg}(1-e^{tz}) - \operatorname{Arg}(-tz) \in (-\pi, \pi) \quad \text{for all } t \in (0, 1].
\]
This implies that the condition \eqref{eq:cond} holds, i.e., $N(t)=0$, and therefore the relation is valid for all $z \in \mathscr{D} \setminus [0, \infty)$ and $t \in (0, 1]$. This justifies writing
\[
\log\Bigl(\frac{1-e^{tz}}{-tz}\Bigr)
=\log\bigl(1-e^{tz}\bigr)-\log(-z)-\log t,
\]
and thus
\[
\int_0^1\log\Bigl(\frac{1-e^{tz}}{-tz}\Bigr)dt
=\int_0^1\log\bigl(1-e^{tz}\bigr)dt-\log(-z)-\int_0^1\log t\,dt.
\]
Since \(\int_0^1\log t\,dt=-1\), it follows that
\[
\int_0^1\log\Bigl(\frac{1-e^{tz}}{-tz}\Bigr)dt
=\int_0^1\log\bigl(1-e^{tz}\bigr)dt-\log(-z)+1.
\]
Now, for \(\Re z < 0\) one may expand
\[
\log\bigl(1-e^{tz}\bigr)=-\sum_{k\ge1}\frac{e^{zkt}}{k},
\]
so that
\[
\int_0^1\log\bigl(1-e^{tz}\bigr)dt
=-\sum_{k\ge1}\frac{1}{k}\int_0^1e^{zkt}\,dt
=-\frac{1}{z}\sum_{k\ge1}\frac{e^{zk}-1}{k^2}.
\]
Hence,
\[
\int_0^1\log\bigl(1-e^{tz}\bigr)dt
=-\frac{1}{z}\Bigl(\operatorname{Li}_2(e^z)-\frac{\pi^2}{6}\Bigr).
\]
Thus, the identity holds for all $z$ with $\Re z < 0$. Let $\phi(z) := \int_0^1 \log\Bigl(\frac{e^{tz}-1}{tz}\Bigr)dt$ and $R(z) := -\frac{1}{z}\Bigl(\operatorname{Li}_2(e^z)-\frac{\pi^2}{6}\Bigr)-\log(-z)+1$. As shown in the proof of Lemma \ref{lem: properties of eta_f}, the function $\phi$ is analytic on the domain $\mathscr{D} = \{z\in \C: \frac{e^{tz}-1}{tz}\not\in (-\infty, 0] \text{ for all } t\in [0,1]\}$. The function $R$ involves the principal branches of the dilogarithm and the logarithm. Both functions are analytic on the connected open set $\mathscr{D} \setminus [0, \infty)$. Since they agree on the set $\{ z \mid \Re z < 0 \}$, the Identity Theorem for analytic functions, \cite[Theorem 3.2.6]{ablowitz2003complex}, implies that $\phi(z) = R(z)$ for all $z \in \mathscr{D} \setminus [0, \infty)$. By continuity and analytic extension, the equality extends for all $z \in \mathscr{D}$. 
\end{proofof}

\begin{remark}
Note that Equation \eqref{eq:log_relation} does not hold in general. Let $t=1$ and $z = \pi + 2\pi i$. In this case $\operatorname{Arg}(e^z-1) - \operatorname{Arg}(z) \approx 5.176 > \pi$ and thus the product rule for $\log$ is violated. In this example we have $z \notin \mathscr{D}$.
\end{remark}

\begin{proofof}{Theorem {\ref{thm:modphi-colored}}}
Apply Theorem \ref{thm:modphi-mixed} with \(a_{N,j}\equiv1\), \(b_{N,j}=2j\) and \(s_N=N\). Then \(b_{N,j}/s_N=f(j/N)\) with \(f(t)=2t\in C^2[0,1]\). Also \(\delta_{N,j}=a_{N,j}/s_N=1/N\to0\) and
\[
\Sigma=\lim_{N\to\infty}\sum_{j=1}^N\frac{a_{N,j}}{N}
=\lim_{N\to\infty}\frac{1}{N}\sum_{j=1}^N 1 = 1.
\]
Therefore the hypotheses of Theorem~\ref{thm:modphi-mixed} hold, yielding the stated \(\phi\) and \(\psi\).
\end{proofof}

\begin{proofof}{Theorem {\ref{thm: mod_phi_j_squared}}}
We use Theorem~\ref{thm:modphi-mixed} with $s_N=N^2$, $a_{N,j}=j$ and $b_{N,j}=j^2$ for $j=1,\ldots,N$.  Then $f(t)=t^{2}$ and
\[
\phi(z)=\int_{0}^{1} \log\Bigl(\frac{1-e^{zt^2}}{-zt^2}\Bigr)\,dt.
\]
Further, $\delta_{N,j}=j/N^{2}\to0$, as $N\to\infty$, and
\[
\Sigma = \lim_{N\to\infty}\sum_{j=1}^{N}\frac{j}{N^{2}} = \lim_{N\to\infty} \frac{1}{N^2} \frac{N(N+1)}{2} = \frac12.
\]
Lemma~\ref{lem: asymptotics Taylor} therefore yields $S_N^a(z)=\frac{z}{2}\Sigma+o(1)=\frac{z}{4}+o(1)$, and hence we arrive at
\[
\psi(z)=\exp\Bigl(\frac12 \log\Bigl(\frac{1-e^{z}}{-z}\Bigr)-\frac{z}{4}\Bigr)=\sqrt{\frac{e^{z}-1}{z}}\,e^{-z/4}.
\]
This shows that $(X_N/N^2)_{N\in\N}$ converges mod-$\phi$ with parameter $N$, normalization function $\phi$ and limiting function $\psi$.
\end{proofof}

\begin{proofof}{Theorem {\ref{thm:modphi-multiset-inversions}}}
We apply the general framework of Theorem~\ref{thm:modphi-mixed} with scaling $s_N=n_N$ and numerator parameters $b_{N,j}=j$ for $j=1,\ldots,n_N$. This choice identifies the profile function as $f(t)=t$ for $t \in [0,1]$, since $b_{N,j}/s_N = j/n_N$. 

Next, we consider the parameters $a_{N,j}$ appearing in the denominator of the $q$-multinomial formula. The assumption $M_N = o(n_N)$ ensures that the individual terms $\delta_{N,j} := a_{N,j}/n_N$ vanish uniformly, that is,
\[
\max_{1 \le j \le n_N} \delta_{N,j} = \frac{M_N}{n_N} \longrightarrow 0 \quad \text{as } N \to \infty.
\]
Furthermore, the aggregated shift constant $\Sigma$ is calculated by summing the internal arithmetic progressions associated with each letter type:
\[
\Sigma = \lim_{N\to\infty} \sum_{j=1}^{n_N} \frac{a_{N,j}}{n_N} = \lim_{N\to\infty} \frac{1}{n_N} \sum_{i=1}^{r(N)} \sum_{k=1}^{m_i^{(N)}} k = \lim_{N\to\infty} \frac{1}{n_N} \sum_{i=1}^{r(N)} \frac{m_i^{(N)}(m_i^{(N)}+1)}{2}.
\]
By the general form of the limiting function $\psi$ provided in Theorem \ref{thm:modphi-mixed}, the boundary values $f(0)=0$ and $f(1)=1$ combined with the shift $\Sigma$ yield
\[
\psi(z) = \exp\left( \frac{1}{2} \log\left( \frac{e^z - 1}{z} \right) - \frac{z}{2} \Sigma \right) = \sqrt{\frac{e^z - 1}{z}} e^{-\frac{z}{2}\Sigma}.
\]
The normalization function $\phi$ and the domain $\mathscr{D}$ follow directly from the linear profile $f(t)=t$, which matches the classical inversion case. This completes the argument.
\end{proofof}

\acknowledgements We would like to thank Zakhar Kabluchko (M\"unster) for enlightening discussions and ideas about inversions of permutations. 

CT was supported by the Deutsche Forschungsgemeinschaft (DFG, German Research Foundation) through SPP 2458 \emph{Combinatorial Synergies}.

PT was supported by the DFG project \emph{Limit theorems for the volume of random projections of $\ell_p$-balls} (project number 516672205).

\begingroup
\newcommand{\n}[1]{\textsc{#1}}
\providecommand{\and}{and~}
\renewcommand{\it}{\itshape}

\endgroup

\end{document}